\documentclass[12pt]{article}
\usepackage{amsthm,amssymb}

\newtheorem{theorem}{Theorem}
\newtheorem{corollary}[theorem]{Corollary}

\newtheorem{lemma}[theorem]{Lemma}

\def\irr#1{{\rm  Irr}(#1)}
\def\ibr#1{{\rm IBr} (#1)}

\def\cent#1#2{{\bf C}_{#1}(#2)}

\def\norm#1#2{{\rm N}_{#1} (#2)}

\def\B#1#2{{\rm B}_{#1} (#2)}
\def\Bpi#1{\B {\pi}{#1}}

\def\I#1#2{{\rm I}_{#1} (#2)}
\def\Ipi#1{\I {\pi}{#1}}
\def\phi{\varphi}

\newcommand \IIpi[3] {{\rm I}_{#1} (#2 \mid #3)}

\begin{document}

\title{Inducing $\pi$-partial characters with a given vertex}

\author {
       Mark L.\ Lewis
    \\ {\it Department of Mathematical Sciences, Kent State University}
    \\ {\it Kent, Ohio 44242}
    \\ E-mail: lewis@math.kent.edu
       }

\maketitle

\begin{abstract}
Let $G$ be a solvable group.  Let $p$ be a prime and let $Q$ be a
$p$-subgroup of a subgroup $V$.  Suppose $\phi \in \ibr G$.  If
either $|G|$ is odd or $p = 2$, we prove that the number of Brauer
characters of $H$ inducing $\phi$ with vertex $Q$ is at most $|\norm
GQ: \norm VQ|$.

MSC Primary: 20C20, MSC Secondary: 20C15

\end{abstract}



\section{Introduction}

Throughout this note, $G$ is a finite group, and $\irr G$ is the set
of irreducible characters of $G$.  Suppose $\chi \in \irr G$ and $H$
is a subgroup of $G$.  It is easy to obtain an upper bound on the
number of characters in $\irr H$ that induce $\chi$.  Let $\phi_1,
\dots, \phi_n \in \irr H$ be the characters so that $\phi_i^G =
\chi$.  Evaluating at $1$, we obtain $\phi_i (1) = \chi (1)/|G:H|$
for each $i$.  By Frobenius reciprocity (Lemma 5.2 of \cite{text}),
each $\phi_i$ is a constituent of $\chi_H$ with multiplicity $1$.
Since there are $n$ such characters occurring as constituents of
$\chi_H$, it follows that $n (\chi (1)/|G:H|) \le \chi (1)$.  We
deduce that $n \le |G:H|$, and we have an upper bound.  If $H$ is
normal in $G$, this bound is obtained, and it is not particularly
difficult to find nonnormal subgroups where this bound is obtained.

We now turn our attention to Brauer characters.  Fix a prime $p$. We
will write $\ibr G$ for the irreducible $p$-Brauer characters of
$G$.  If $\phi \in \ibr G$, then it is easy to adapt the above proof
to show that $\phi$ is induced by at most $|G:H|$ Brauer characters
of $H$.  However, associated with $\phi$ are certain $p$-subgroups
of $G$ called the vertex subgroups.  When $G$ is a $p$-solvable
group, a $p$-subgroup $Q$ of $G$ is defined to be a vertex for
$\phi$ if there is a subgroup $U$ of $G$ so that $\phi$ is induced
by a Brauer character of $U$ with $p'$-degree and $Q$ is a Sylow
subgroup of $U$.  It is known that all the vertex subgroups of
$\phi$ are conjugate in $G$. If $\phi$ is induced from $\tau \in
\ibr H$, it is easy to see that a vertex for $\tau$ is a vertex for
$\phi$.  Thus, $H$ contains some vertex $Q$ for $\phi$.  Now,
different Brauer characters of $H$ that induce $\phi$ may have
vertex subgroups that are not conjugate in $H$ but are necessarily
conjugate in $G$. Hence, one can ask the following question: Suppose
$\phi \in \ibr G$ has vertex $Q$, and $Q \le H$, how many characters
in $\ibr H$ with vertex $Q$ induce $\phi$?  When either $|G|$ is odd
or $G$ is solvable and $p = 2$, we can obtain an upper bound for
this question.

\begin{theorem} \label{main}
Let $G$ be a solvable group and $p$ a prime.  Assume either $|G|$ is
odd or $p = 2$. Let $Q$ be a $p$-subgroup of $H$. If $\phi \in \ibr
G$, then the number of Brauer characters of $H$ with vertex $Q$ that
induce $\phi$ is at most $|\norm GQ:\norm HQ|$.
\end{theorem}

At this time, we are not able to determine whether or not this
theorem is true if we loosen the hypothesis that either $|G|$ is odd
or $p = 2$.  In other words, is the conclusion still true if $G$ is
a solvable group of even order and $p$ is an odd prime. This result
was motivated by our work with J. P. Cossey. If we could prove the
conclusion of Theorem \ref{main} when $p$ is odd, then we would be
able to prove J. P. Cossey's conjecture that the number of lifts of
a Brauer character is bounded by the index of a vertex subgroup in
the vertex subgroup when $p$ is odd. Our argument can be found in
the preprint \cite{preprint}.

We would like to thank J. P. Cossey and I. M. Isaacs for several
helpful discussions while we were preparing this note.

\section{Results}

We will in the more general setting of irreducible $\pi$-partial
characters of a $\pi$-separable group $G$.  We here briefly mention
that if $\pi$ is a set of primes and $G$ is a $\pi$-separable group,
one can define (see \cite{pipart} for more details) a set of class
functions $\Ipi G$ from the set $G^o$ (which consists of the
elements of $G$ whose order is divisible by only the primes in
$\pi$) to ${\bf C}$ that plays the role of $\ibr G$, and in fact
$\Ipi G = \ibr G$ if $\pi = \{ p' \}$, the complement of the prime
$p$.

%

We start by considering vertices in Clifford correspondence (see
Proposition 3.2 of \cite{Fong}).  Let $G$ be a $\pi$-separable
group. Let $N$ be a normal subgroup of $G$.   Fix $\phi \in \Ipi G$.
If $\alpha \in \Ipi N$ is a constituent of $\phi_N$, then we write
$G_\alpha$ for the stabilizer of $\alpha$ in $G$, and we write
$\phi_\alpha$ for the Clifford correspondent of $\phi$ with respect
to $\alpha$. In particular, the vertices of the Clifford
correspondent form an orbit under the action of the normalizer of a
particular vertex.

\begin{lemma}\label{cliff}
Let $G$ be a $\pi$-separable group.  Let $N$ be a normal subgroup of
$G$.  Suppose that $\alpha \in \Ipi N$.  Let $\phi \in \Ipi G$ and
$\hat\phi \in \Ipi {G_\alpha}$ so that $\hat\phi^G = \phi$.  Suppose
that $Q$ is a vertex for $\hat\phi$.  Then $Q$ is a vertex is
$\hat\phi^g$ if and only if there exists $n \in \norm GQ$ so that
$G_\alpha g = G_\alpha n$.
\end{lemma}

\begin{proof}
We first suppose that there exists $n \in \norm GQ$ so that
$G_\alpha g = G_\alpha n$.  Thus, $g = tn$ for some $t \in
G_\alpha$.  We see that $\hat\phi^g = \hat\phi^{tn} = \hat\phi^n$.
We see that $Q = Q^n$ is a vertex for $\hat\phi^n = \hat\phi^g$.

Conversely, suppose that $Q$ is a vertex for $\hat\phi^g$.  Then
$Q^{g^{-1}}$ is a vertex for $\hat\phi$.  Since $Q$ is also a vertex
for $\hat\phi$, we have $Q^{g^{-1}} = Q^t$ for some $t \in
G_\alpha$.  It follows that $Q = Q^{tg}$, and so, $tg \in \norm GQ$.
This implies that $tg = n$ for some $n \in \norm GQ$.  This implies
that $n \in G_\alpha g$, and we conclude that $G_\alpha n = G_\alpha
g$.
\end{proof}

We continue to work in the context of the Clifford correspondence.
In this case, we can get an exact count of the number of partial
characters in $N$ whose Clifford correspondent has vertex $Q$.

\begin{corollary}\label{cliff count}
Let $G$ be a $\pi$-separable group.  Let $N$ be a normal subgroup of
$G$, let $\phi \in \Ipi G$ have vertex $Q$, and suppose that $\beta$
is an irreducible constituent of $\phi_N$ so that $\phi_\beta$ has
vertex $Q$. Then $|\{ \alpha \in \Ipi N \mid \phi_\alpha {\rm ~
has~vertex~} Q \}| = |\norm GQ : \norm {G_\beta}Q|$.
\end{corollary}

\begin{proof}
By Lemma \ref{cliff}, we see that $\phi_\alpha$ has $Q$ as a vertex
if and only if $\alpha = \beta^g$ where $g \in G$ and $g \in G_\beta
n$ for some $n \in \norm GQ$. Finally, we observe that $G_\beta n_1
= G_\beta n_2$ if and only if $\norm {G_\beta}Q n_1 = \norm
{G_\beta}Q n_2$ for $n_1, n_2 \in \norm GQ$.  We have $|\{ \alpha
\in \Ipi N \mid \phi_\alpha {\rm ~ has~vertex~} Q \}| = | \{ G_\beta
n \mid n \in \norm GQ \} | = |\norm GQ :\norm {G_\beta}Q|$.
\end{proof}

We now look at the conditions of a minimal counterexample.  For this
we need to review and develop more notation.  We make use of the
canonical set of $\pi$-lifts, $\Bpi G$, that was defined in
\cite{pisep} by Isaacs.  In other words, $\Bpi G \subseteq \irr G$
and the map $\chi \mapsto \chi^o$ is a bijection from $\Bpi G$ to
$\Ipi G$.  Closely related to this set is the subnormal nucleus
which also was defined in \cite{pisep}.  To define the subnormal
nucleus, we need the $\pi$-special characters.  Let $G$ be a
$\pi$-separable group.  A character $\chi \in \irr G$ is
$\pi$-special if $\chi (1)$ is a $\pi$-number and for every
subnormal group $M$ of $G$, the irreducible constituents of $\chi_M$
have determinants that have $\pi$-order.  Many of the basic results
of $\pi$-special characters can be found in Section 40 of
\cite{hupte} and Chapter VI of \cite{Mawo}. One result that is
proved is that if $\alpha$ is $\pi$-special and $\beta$ is
$\pi'$-special, then $\alpha \beta$ is necessarily irreducible. We
say that $\chi$ is {\bf factored} if $\chi = \alpha \beta$ where
$\alpha$ is $\pi$-special and $\beta$ is $\pi'$-special.  We also
note that if $\chi \in \Bpi G$ and $N$ is normal in $G$, then the
irreducible constituents of $\chi_N$ lie in $\Bpi N$.





If $\chi \in \irr G$, Isaacs constructs the subnormal vertex as
follows.  Let $M$ be maximal so that $M$ is subnormal in $G$ and the
irreducible constituents of $\chi_M$ are factored.  Let $\mu$ be an
irreducible constituent of $\chi_M$ and let $T$ be the stabilizer of
$(M,\mu)$ in $G$.  Isaacs proved in \cite{pisep} that there is a
Clifford theorem for $T$.  In other words, there is a unique
character $\tau \in \irr {T \mid \mu}$ so that $\tau^G = \chi$.  He
also proved that $(M,\mu)$ is unique up to conjugacy, and so,
$(T,\tau)$ is unique up to conjugacy.  If $T = G$, then $\chi$ is
$\pi$-factored and we take $(G,\chi)$ to be the subnormal nucleus of
$\chi$.  If $T < G$, then inductively, the subnormal nucleus for
$\tau$ is the subnormal nucleus for $\chi$.  We write $(W,\gamma)$
for the subnormal nucleus of $\chi$, and Isaacs showed that
$\gamma^G = \chi$, $\gamma$ is factored, and $(W,\gamma)$ is unique
up to conjugacy.  A character $\chi \in \irr G$ is in $\Bpi G$ if
and only if the character of its nucleus is $\pi$-special.

If $Q$ is a $\pi'$-subgroup of $G$, then we use $\Ipi {G \mid Q}$ to
denote the $\pi$-partial characters in $\Ipi G$ that have vertex
$Q$.  If $\phi \in \Ipi G$ and $V \le G$, then we write $\IIpi
{\phi}VQ = \{ \eta \in \Ipi {V \mid Q} \mid \eta^G = \phi \}$.  We
now find details about properties of a minimal counterexample.  We
will see that a counterexample cannot occur when either $|G|$ is odd
or $2$ is not in $\pi$.  Our goal is find enough information so that
we can either find a contradiction or build an example when $|G|$ is
even and $2 \in \pi$.

\begin{theorem} \label{min counter}
Let $G$ be a solvable group.  Assume $\phi \in \Ipi G$ has vertex
$Q$, let $V$ be a subgroup of $G$, and let $N$ be the core of $V$ in
$G$. If $G$ and $V$ are chosen so that $|G| + |G:V|$ is minimal
subject to the condition that $|\IIpi {\phi}VQ| > |\norm GQ:\norm
VQ|$, then the following are true:

\begin{enumerate}
\item $V$ is a nonnormal maximal subgroup of $G$,
\item $|G:V|$ is a power of $2$,
\item $2 \in \pi$,
\item $Q \le V$,
\item $\phi_N = a \alpha$ for some $\alpha \in \ibr N$,
\item $\alpha (1)$ is a $\pi$-number,
\item if $K$ is normal in $G$ so that $K/N$ is a chief factor for
$G$, then $\alpha$ is fully ramified with respect to $K/N$.
\end{enumerate}
\end{theorem}

\begin{proof}
If either $V = G$ or $\IIpi {\phi}VQ$ is empty, then $|\IIpi
{\phi}VQ| \le |\norm GQ:\norm VQ|$ contradicting the hypotheses.
Thus, $V < G$ and $\IIpi {\phi}VQ$ is not empty, and so, $Q \le V$
and there exist characters in $\Ipi V$ that induce $\phi$ and have
vertex $Q$.

We begin by showing that $V$ is a maximal subgroup.  Suppose that $V
< M < G$ for some subgroup $M$.  Let $\IIpi {\phi}MQ = \{ \eta_1,
\dots, \eta_m \}$.  Using minimality, we have $m = |\IIpi {\phi}MQ|
\le |\norm GQ:\norm MQ|$. Suppose that $\zeta \in \IIpi {\phi}VQ$,
then $\zeta^M \in \Ipi M$ and $\zeta^M$ has $Q$ as a vertex.  Since
$(\zeta^M)^G = \zeta^G = \phi$, we see that $\zeta^M \in \IIpi
{\phi}MQ$.  It follows that $\zeta^M = \eta_i$ for some $i$.  We
conclude that $|\IIpi {\phi}VQ| \le \sum_{i=1}^m |\IIpi
{\eta_i}VQ|$.  Since this contradicts our hypothesis, we obtain
$|\IIpi {\eta_i}VQ| \le |\norm MQ:\norm VQ|$.  We deduce that
$$
|\IIpi {\phi}VQ| \le m|\norm MQ:\norm VQ| \le
|\norm GQ:\norm MQ||\norm MQ: \norm VQ| = |\norm GQ:\norm VQ|.
$$
Since this violates the hypotheses, $V$ is maximal in $G$.

If $V$ is normal in $G$, then either $\phi$ is induced from $V$ or
$\phi$ restricts irreducibly to $V$.  If $\phi$ is induced from $V$,
then we can apply Corollary \ref{cliff count} to see that $|\IIpi
{\phi}VQ| \le |\norm GQ:\norm VQ|$ in violation of the hypotheses.
If $\phi$ restricts irreducibly, then it cannot be induced from $V$,
and we have seen that this is also a contradiction.  We conclude
that $V$ is not normal in $G$.

Suppose $\alpha \in \Ipi N$ is a constituent of $\phi_N$.  We use
$\phi_\alpha \in \Ipi {G_\alpha \mid \alpha}$ to denote the Clifford
correspondent for $\phi$ with respect to $\alpha$ (see Proposition
3.2 of \cite{Fong} again).   Write $\{ \alpha \in \Ipi N \mid
\phi_\alpha {\rm ~has~vertex~} Q \} = \{ \alpha_1, \dots, \alpha_k
\}$, and let $\phi_i = \phi_{\alpha_i}$ and $G_i = G_{\alpha_i}$. By
Lemma \ref{cliff count}, we know that $k = |\norm GQ:\norm {G_i}Q|$.

Suppose $\eta \in \IIpi {\phi}VQ$.  Denote $\{ \beta \in \Ipi N \mid
\eta_\beta {\rm ~has~vertex~} Q \} = \{ \beta_1, \dots, \beta_l \}$,
and let $\eta_j = \eta_{\beta_j}$ and $V_j = V_{\beta_j}$.  By Lemma
\ref{cliff count}, $l = |\norm VQ: \norm {V_i}Q|$.

We see that $(\eta_j)^G = ((\eta_j)^V)^G = \eta^G = \phi$.  This
implies that $(\eta_j)^{G_{\beta_j}}$ is irreducible and has vertex
$Q$.  It follows that $\beta_j = \alpha_{i_j}$ for some $i_j$.  We
obtain $G_{\beta_j} = G_{i_j}$ and $(\beta_j)^{G_{i_j}} =
\alpha_{i_j}$. Observe that $V_j = G_{i_j} \cap V$, and we denote
this subgroup by $V^*_{i_j}$.

Now, we assume that $k > 1$, and we start to count.  We see that
$\eta \in \IIpi {\phi}GQ$ is induced by $|\norm VQ:\norm
{V^*_{i_j}}Q|$ partial characters in $\bigcup \IIpi
{\phi_i}{V^*_i}Q$.  Because $G_i < G$, we may use minimality of $|G|
+ |G:V|$ to deduce $|\IIpi {\phi_i}{V^*_i}Q| \le |\norm {G_i}Q:\norm
{V^*_i}Q|$. We compute
$$
|\IIpi {\phi}VQ| = \sum_{i=1}^k \frac 1{|\norm VQ:\norm {V^*_i}Q|}
|\IIpi {\phi_i}{V^*_i}Q| \le \sum_{i=1}^k \frac 1{|\norm VQ:\norm
{V^*_i}Q|} |\norm {G_i}Q:\norm {V^*_i}Q|.
$$
We determine that
$$
\frac 1{|\norm VQ:\norm {V^*_i}Q|} |\norm {G_i}Q:\norm {V^*_i}Q| =
\frac {|\norm {G_i}Q|}{|\norm VQ|},
$$
for each $i$.   Notice that
$|\norm {G_i}Q| = |\norm {G_1}Q$ for all $i$ and $k = |\norm GQ:
\norm {G_1} Q$. This yields
$$
|\IIpi {\phi}VQ| \le \sum_{i=1}k \frac {|\norm {G_1}Q}{\norm VQ} =
\frac {|\norm GQ:\norm {G_1}Q| |\norm {G_1}Q|}{|\norm VQ|} = |\norm
GQ:\norm VQ|.
$$
This contradicts the hypothesis.  We deduce that $k = 1$, and
$\alpha$ is invariant in $G$.

Set $\alpha = \alpha_1$, and let $\alpha^*$ be the character in
$\Bpi N$ satisfying $(\alpha^*)^o = \alpha$.  Write $(W,\hat\alpha)$
for the nucleus of $\alpha^*$, and take $T$ to be the stabilizer of
$(W,\hat\alpha)$ in $G$. By Lemma 2.3 of \cite{Laradji}, there is a
unique character $\hat\phi \in \IIpi {}T{\hat\alpha}$ so that
$\hat\phi^G = \phi$ and $Q$ is a vertex for $\hat\phi$.  Similarly,
if $\eta \in \IIpi {\phi}VQ$, then there is a unique character
$\hat\eta \in \IIpi {}{T \cap V}{\hat\alpha}$ so that $\hat\eta^V =
\eta$ and $Q$ is a vertex for $\hat\eta$. Observe that $\hat\eta^T
\in \IIpi {}T{\hat\alpha}$ and induces $\phi$, so $\hat\eta^T =
\hat\phi$.  It follows that $|\IIpi {\phi}VQ| = |\IIpi {\hat\phi}{T
\cap V}Q|$.  If $T < G$, then we can use the minimality of $|G| +
|G:V|$ to see that $|\IIpi {\hat\phi}{T \cap V}Q| \le |\norm TQ:
\norm {V \cap T}Q|$.  By the diamond lemma, we have $|\norm TQ:
\norm {V \cap T}Q| = |\norm TQ: V \cap \norm TQ| \le |\norm GQ :
\norm VQ|$.  This contradicts the hypotheses, and so $T = G$.

We now have that $(W,\hat\alpha)$ is $G$-invariant.  By the
construction of the subnormal, this implies that $W = N$.  Since
$\alpha^* \in \Bpi N$, the nucleus for $\alpha^*$ has a character
that is $\pi$-special.  Thus, $\hat\alpha$ is $\pi$-special, and
since $W = N$, we see that $\hat\alpha = \alpha^*$.  In particular,
$\hat\alpha$ is $\pi$-special.  We deduce that $\alpha (1)$ is a
$\pi$-number.

Take $K$ normal in $G$ so that $K/N$ is a chief factor for $G$. This
is the point where we use the fact that $G$ is solvable to see that
$G = VK$ and $V \cap K = N$ where $K/N$ is an elementary abelian
$p$-group for some prime $p$.  (This is the only place we use the
hypothesis that $G$ is solvable in place of $G$ being
$\pi$-separable.)
%
%
Let $L/K$ be a chief factor for $G$.  We know that $(|L:K|,|K:N|) =
1$ and $\cent {L \cap V/N}{K/N}$.  (See Lemma 5.1 of \cite{max} for
a proof of this.)  By Problem 6.12 of \cite{text}, either $\alpha^*$
extends to $K$ or $\alpha^*$ is fully-ramified with respect to
$K/N$.  .

Suppose first that $\alpha^*$ extends to $K$. Notice that
multiplication by $\irr {K/L}$ is a transitive action on the
irreducible constituents of $(\alpha^*)^K$.  Also, $(V \cap K)/L$
acts on compatibly on the irreducible constituents of $(\alpha^*)^K$
and on $\irr {K/L}$ where the action on $\irr {K/L}$ is coprime.  We
can use Glauberman's lemma (Lemma 13.8 of \cite{text}) to see that
$\alpha^*$ has a $V \cap L$-invariant extension.  The corollary to
Glauberman's lemma (Corollary 13.9 of \cite{text}) can be applied to
see that $\alpha^*$ has a unique $V \cap L$-invariant extension
$\delta$. Since $V$ permutes the $V \cap L$-extensions of
$\alpha^*$, it follows that $\delta$ is $V$-invariant.  We now use
Corollary 4.2 of \cite{pisep} to see that restriction is a bijection
from $\irr {G \mid \beta}$ to $\irr {V \mid \alpha^*}$.

Let $\eta \in \IIpi {\phi}VQ$ so that $\eta^G = \phi$.  We can find
$\eta^* \in \Bpi V$ so that $(\eta^*)^o = \eta$. Since
$({\eta^*}^G)^o = ({\eta^*}^o)^G = \eta^G = \phi \in \ibr G$, we see
that $\eta^G$ is irreducible.  On the other hand, $({\eta^*}^o)_N =
(\eta_N)^o = b \alpha$ for some integer $b$.  Since the irreducible
constituents of ${\eta^*}_N$ lie in $\Bpi N$, we deduce that $\eta^*
\in \irr {V \mid \alpha^*}$.  But we saw that this implies that
$\eta^*$ extends to $G$.  Since $V < G$, it is not possible for
$\eta^*$ to both extend to $G$ and induce irreducibly.  Therefore,
we have a contradiction.  We see that $\alpha^*$ (and hence,
$\alpha$) is fully ramified with respect to $K/N$.  Notice that if
$p$ is not in $\pi$, then Corollary 6.28 of \cite{text} applies and
$\alpha^*$ extends to $K$. Therefore, $p \in \pi$.

We suppose that $p$ is odd, and we work for a contradiction.  Since
$\alpha^*$ is fully-ramified with respect to $K/N$ and $|K:N|$ has
odd order, main theorem of \cite{fram} implies that no character in
$\irr {V \mid \alpha}$ induces irreducibly to $G$.  (A stronger
theorem is proved in \cite{brown}.) As in the previous paragraph,
this implies that $\phi$ is not induced from $V$ which contradicts
the assumption that $\IIpi {\phi}VQ$ is not empty. (This strongly
uses the fact that $p$ is odd. When $p = 2$, it is tempting to try
use the correspondence in \cite{strong}, but that correspondence
does not preclude inducing characters in $\irr {G \mid \alpha}$ from
$V$. In fact, ${\rm GL}_2 (3)$ is an example where this occurs.) We
conclude that $p = 2$. Since $|G:V| = |K:N|$, we see that $|G:V|$ is
a power of $2$. This proves the theorem.
\end{proof}

As a corollary, we obtain Theorem \ref{main} stated for
$\pi$-partial characters.

\begin{corollary}
Let $G$ be a solvable group.  Assume either $|G|$ is odd or $2
\not\in \pi$.  Let $Q$ be a $\pi'$-subgroup of $G$ and suppose that
$Q \le V$.  If $\phi \in \Ipi G$, then $|\IIpi {\phi}VQ| \le |\norm
GQ:\norm VQ|$.
\end{corollary}

\begin{proof}
We suppose the result is not true.  Let $G$ be a counterexample with
$|G| + |G:V|$ as in Theorem \ref{min counter}.  By that result, we
have that $|G:V|$ is a nontrivial power of $2$ which is a
contradiction if $|G|$ is odd.  We also have $2 \in \pi$ which is a
contradiction to $2 \not\in \pi$.  This proves the corollary.
\end{proof}

\end{document}